\documentclass{amsart}
\usepackage{color}
\usepackage[margin=1in]{geometry}
\usepackage{graphicx,amsthm,amssymb,tikz,bbold,comment}
\usepackage{cite}
\usetikzlibrary{cd}

\usepackage{xcolor}
\colorlet{darkgreen}{green!40!black}
\usepackage{ifpdf}
\ifpdf \usepackage[pdftex]{hyperref}
\else \usepackage[ps2pdf]{hyperref} 
\usepackage{breakurl}
\fi
\hypersetup{colorlinks=true,urlcolor=blue,citecolor=darkgreen,
	linkcolor=darkgreen,linktocpage,bookmarksnumbered,unicode}
\newcommand{\Z}{\mathbb{Z}}\newcommand{\Q}{\mathbb{Q}}\newcommand{\C}{\mathbb{C}}
\newcommand{\e}{\varepsilon}
\newcommand{\Ker}{\textnormal{Ker}}

\newcommand{\tn}[1]{\textnormal{#1}}

\renewcommand{\sp}{\supseteq}
\newcommand{\ra}{\rightarrow}

\newcommand{\MTU}{\mathbf{MTU}}\newcommand{\MU}{\mathbf{MU}}\newcommand{\Ext}{\mathbf{Ext}}\newcommand{\Hom}{\mathbf{Hom}}

\newtheorem{remark}{Remark}[section]

\newtheorem{definition}{Definition}[section]

\newtheorem{theorem}{Theorem}[section]
\newtheorem{proposition}[theorem]{Proposition}

\newtheorem{corollary}[theorem]{Corollary}
\newtheorem{lemma}[theorem]{Lemma}
\begin{document}
\title{Cobordism Obstructions to Complex Sections II: Torsion Obstructions}
\author{Dennis Nguyen}
\email{dpn@uoregon.edu}
\address{ Department of Mathematics, University of Oregon, Fenton Hall, 1021 E 13th Ave, Eugene, OR, 97403 }


\begin{abstract}
	
In the previous paper, we studied obstructions to the existence of complex sections on almost complex manifolds up to cobordism. We determined the obstruction rationally, in terms of the Chern classes.  In this paper, we study the torsion obstructions, that is, the obstructions which vanish after tensoring with $\Q$ or multiplication by an integer. Calculations with the Adams-Novikov spectral sequence for the Thom spectra $\MTU(d)$ allow us to show the torsion obstructions for low $r$. For prime $p\geq 3$, we show that torsion obstructions for finding $r$ complex sections of order $p$ vanish for $r<p^2-p$. 

\end{abstract}
\maketitle
\section{Introduction}

There is a classical problem to determine whether a manifold admits $r$
linearly independent tangent vector fields. In the work of B\"okstedt, Dupont and Svane \cite{bokstedt} and in the previous paper \cite{nguyen1}, we approached this problem by instead determining the obstruction to finding a cobordant manifold with $r$ vector fields. We were able to describe the rational obstruction to $r$ linearly independent complex sections of the tangent bundle of an almost complex manifold up to cobordism. This obstruction is given in terms of Chern characteristic numbers as follows:

\begin{theorem}\label{ratobs}
	Let $M^{2d}$ be a $d$-dimensional almost-complex manifold. The Chern characteristic classes $s_\omega(M^{2d})=0$ for all partitions $\omega$ of length greater than $d-r$ if and only if  there exists integer $c$ such that the cobordism class $c[M^{2d}]$ contains a manifold $N^{2d}$ with $r$ complex sections on $TN$.
\end{theorem}

The length of a partition is the number of summands in the partition. For any cobordism class $[M^{2d}]$ with vanishing rational obstruction, we can ask whether there are any divisibility conditions for the constant $c$ in the above theorem. If, for all $[M^{2d}]\in \Omega_{2d}^U$ with vanishing rational obstruction, the constant $c$ may be chosen to be $1$ mod $p$ for prime $p$, we say that the $p$ torsion obstruction vanishes. This is equivalent to a particular map having no $p$-torsion in its image, see Theorem 4.6 in \cite{nguyen1}. We will show in this paper, Proposition \ref{diff}, that $p$ torsion obstructions correspond to non-vanishing differentials in the $p$-primary Adams-Novikov spectral sequence for $\MTU(d-r)$. The prime $p>2$ case is tackled first:

\begin{theorem}[Theorem \ref{torsobsbody}]\label{torsobs}
	For $p>2$, the $p$-torsion obstruction to the existence of $r$ complex sections of the tangent bundle of an almost complex manifold vanishes when $r<p^2-p$.
\end{theorem}

We can further show that the $2$-torsion obstruction vanishes for $r=2$ or $r=3$ complex sections. Thus we can prove the main result of the paper:

\begin{theorem}[Theorem \ref{3vf2}]\label{3vf}
		Let $M^{2d}$ be a $d$-dimensional almost-complex manifold with $d\geq 6$. There is a manifold $N^{2d}\in [M^{2d}]$ with $2$ complex sections of $TN$ if and only if $\chi([M^{2d}])=0$ and $s_{2,1,...,1}([M^{2d}])=0$. 
		
		There is a manifold $N^{2d}\in [M^{2d}]$ with $3$ complex sections of $TN$ if and only if $s_{2,2,1...,1}([M^{2d}])=0$, $s_{3,1,...,1}([M^{2d}])=0$, $s_{2,1,...,1}([M^{2d}])=0$, and $\chi([M^{2d}])=0$. 
\end{theorem}

In the case $d$ is odd, we can go one step further:

\begin{theorem}[Theorem \ref{4vf2}]\label{4vfodd}
	If $d>6$ is odd, there is a manifold $N^{2d}\in [M^{2d}]$ with $4$ complex sections of $TN$ if and only if $s_{4,1...,1}([M^{2d}])=0$, $s_{3,2,1...,1}([M^{2d}])=0$, $s_{2,2,2,...,1}([M^{2d}])=0$, $s_{2,2,1...,1}([M^{2d}])=0$, $s_{3,1,...,1}([M^{2d}])=0$, $s_{2,1,...,1}([M^{2d}])=0$, and $\chi([M^{2d}])=0$. 
\end{theorem}

{\bf{Acknowledgments:}} I would like to thank all the people who have helped me as I have completed this paper. First, I would like to thank the University of Oregon Department of Mathematics for supporting me during this work. I most especially thank my advisor, Boris Botvinnik, for the time he spent guiding and advising me throughout this project. I also thank Douglas Ravenel for his feedback and suggestions.

\section{Construction of Spectra}
\subsection{The spectra MTU(d) and MTU(d,r)}
Here we introduce the spectra $\MTU(d)$ and $\MTU(d,r)$. The homotopy groups of these spectra will be the natural objects of study. The real case is described in \cite{bokstedt} and \cite{galatius}. Most of this work is repeated from our previous paper \cite{nguyen1}. It should be noted that there is some disagreement about the proper indexing of these spectra between \cite{bokstedt} and \cite{galatius}; we follow the convention given in B\"okstedt, Dupont and Svane. \cite{bokstedt} Write $G_\C(d,n)$ for the complex Grassmannian of $d$ dimensional complex subspaces of
$\C^{d+n}$. Let $U_{\C, d,n}\ra G_\C(d,n)$ be the tautological $d$
dimensional complex vector bundle and let $U_{\C, d,n}^\perp \ra
G_\C(d,n)$ be the $n$ dimensional orthogonal complement of $U_{\C, d,n}$. 
\begin{definition}
	Define $\MTU(d)$ to be the spectrum whose $2n$-th
	space is $\MTU(d)_{2n}= Th(U_{\C, d,n}^\perp)$.
\end{definition}
There exists a
canonical map $G_\C(d,n)\ra G_\C(d,n+1)$ defined by the
composition: \[\C^d\ra \C^{d+n} \hookrightarrow \C \oplus \C^{d+n} \cong
\C^{d+n+1}\]
The restriction of the bundle $U_{\C, d,n+1}^\perp$ to
$G_\C(d,n)$ under this map is $U_{\C, d,n}^\perp\oplus \C $. (Here $\C$ is a one dimensional complex trivial bundle.) In other words, there is a bundle map $U_{\C, d,n}^\perp\oplus \C\ra U_{\C,
	d,n+1}^\perp$ covering the map $G(d,n)\ra G(d,n+1)$. This map
induces a map of Thom spaces and the following composition gives the spectrum map $\Sigma^2 \MTU(d)_{2n}\ra \MTU(d)_{2n+2}$: 
\[\Sigma^2(Th(U_{\C,d,n}^\perp))\cong S^2\wedge Th(U_{\C, d,n}^\perp) \cong Th(\C \oplus
U_{\C, d,n}^\perp) \ra Th(U_{\C, d,n+1}^\perp)\]
There is also a map $G(d-r,n)\ra G(d,n)$ which takes a $(d-r)$-dimensional complex plane $P\subset \C^{d-r+n}$ to the $d$-plane $P\oplus \C^r \subset
\C^{d-r+n} \oplus \C^r$. Under this map, the pullback of $U_{\C, d,n}^\perp$
is $U_{\C, d-r,n}^\perp$. The pullback induces canonical maps of Thom spaces
$\MTU(d-r)_{2n}\ra \MTU(d)_{2n}$. Since this map commutes with the
spectrum map, it defines a map of spectra
$\MTU(d-r)\ra \MTU(d)$.
\begin{definition}
	Let $\MTU(d,r)$ be the cofiber of the map $\MTU(d-r)\ra \MTU(d)$.
\end{definition}
We note that, having defined $\MTU(d,r)$, we immediately get a cofibration for $k\leq d-r$.
	\[\MTU(d-r,k)\ra \MTU(d,r+k) \ra \MTU(d,r)\]
This cofibration reduces to the definition when $k=d-r$. There is a sequence of spectra:
\[ \MTU(d)\ra \MTU(d+1) \ra \MTU(d+2)\ra ...\]
Let us call the colimit of this spectrum $\MTU$; this spectrum is the Thom spectrum of the bundle $U^\perp_\C$ over $BU$. There is an inversion map $BU\ra BU$ which takes a vector space to its complement. This map is covered by a bundle map $U^\perp_\C\ra U_\C$ which defines a homotopy equivalence $\MTU\ra \MU$. See \cite{ebert} for the real case of this equivalence. In particular, there is a well defined map $\MTU(d)\ra \MU$. 
\begin{definition}
	Let $\overline{\MTU}(d)$ be the cofiber of the map $\MTU(d)\ra \MU$.
\end{definition}

There is an alternative interpretation of $\overline{\MTU}(d)$ as the colimit of the sequence of spectra: \[ \MTU(d+1,1)\ra \MTU(d+2,2) \ra ...\]
The following commutative diagram relates these two constructions and shows that they are homotopy equivalent.
\begin{center}
	
	\begin{tikzpicture}
		\node (v1) at (-3,0.5) {$\MTU(d)$};
		\node (v2) at (0,0.5) {$\MTU(d+1)$};
		\node (v3) at (3,0.5) {$\MTU(d+1,1)$};
		\draw [->] (v1) edge (v2);
		\draw [->] (v2) edge (v3);
		\node (v5) at (0,1.5) {$\MTU(d+2)$};
		\node (v9) at (3,1.5) {$\MTU(d+2,2)$};
		\node (v4) at (0,2.5) {\vdots};
		\node (v8) at (3,2.5) {\vdots};
		
		\node (v6) at (0,3.5) {$\MU$};
		\node (v7) at (3,3.5) {$\overline{\MTU}(d)$};
		\draw [->] (v1) edge (v6);
		\draw [->] (v1) edge (v5);
		\draw [->] (v6) edge (v7);
		\draw [->] (v8) edge (v7);
		\draw [->] (v9) edge (v8);
		\draw [->] (v5) edge (v4);
		\draw [->] (v4) edge (v6);
		\draw [->] (v3) edge (v9);
		\draw [->] (v2) edge (v5);
		\draw [->] (v5) edge (v9);
	\end{tikzpicture}
\end{center}
\begin{proposition}
	There is a cofibration $\MTU(d,r) \ra \overline{\MTU}(d-r) \ra \overline{\MTU}(d)$
\end{proposition}
\begin{proof}
	There is a commutative diagram as below. The columns are the sequences of spectra used to define $\overline{\MTU}(d-r)$ and $\overline{\MTU}(d)$. The rows are each cofibrations.
	\begin{center}
				\begin{tikzpicture}
			\node (v1) at (-3.5,0.5) {$\MTU(d,r)$};
			\node (v2) at (0,0.5) {$\MTU(d+1,r+1)$};
			\node (v3) at (3.5,0.5) {$\MTU(d+1,1)$};
			\draw [->] (v1) edge (v2);
			\draw [->] (v2) edge (v3);
			\node (v5) at (0,1.5) {$\MTU(d+2,r+2)$};
			\node (v9) at (3.5,1.5) { $\MTU(d+2,2)$};
			\node (v4) at (0,2.5) {\vdots};
			\node (v8) at (3.5,2.5) {\vdots};
			\node (v6) at (0,3.5) {$\overline{\MTU}(d-r)$};
			\node (v7) at (3.5,3.5) {$\overline{\MTU}(d)$};
			\draw [->] (v1) edge (v6);
			\draw [->] (v1) edge (v5);
			\draw [->] (v6) edge (v7);
			\draw [->] (v8) edge (v7);
			\draw [->] (v9) edge (v8);
			\draw [->] (v5) edge (v4);
			\draw [->] (v4) edge (v6);
			\draw [->] (v3) edge (v9);
			\draw [->] (v2) edge (v5);
			\draw [->] (v5) edge (v9);
		\end{tikzpicture}
	\end{center}
	The colimit of the cofibrations defined by each row gives the desired cofibration.
\end{proof}
	There is a second construction of spectrum $\MTU(d)$ due to \cite{bokstedt} which connects these spectra to the structure of complex sections. For any $d$-dimensional complex fiber bundle $E\ra X$ equipped with a Hermitian inner product there is a complex frame bundle $V_{\C,r}(E)\ra X$ with fiber $V_{\C,d,r}$. The space  $V_{\C,d,r}$ is the frame manifold of $r$ ordered complex frames in $\C^d$. 
Now consider the specific case where the bundle is $U_{\C,d,n}\ra G_\C(d,n)$. We can pullback the original construction:
\begin{center}
	
	\begin{tikzcd}[column sep=large, row sep=huge]
		
		p_{V_{\C,r}}^* U_{\C,d,n}^\perp\arrow[r]  
		\arrow[d]& U_{\C,d,n}^\perp \arrow[d] \\
		V_{\C,r} (U_{\C,d,n}) \arrow[r, "p_{V_{\C,r}}"] 	& G_\C(d,n)
	\end{tikzcd}
\end{center}

We form the Thom spaces $Th(p_{V_{\C,r}}^* U_{\C,d,n}^\perp)$ into a Thom spectra which we call $\MTU(d)_{V_r}$. Similarly let $W_{\C,r}(U_{\C,d,n} )$ be the bundle whose fiber is the space of ordered $r$-tuples in $U_{\C,d,n}$ which are (hermitian) orthogonal and of the same length which may be between $0$ and $1$. The fiber $W_{\C,d,r}$ is the cone over $V_{\C,d,r}$ by construction. We can construct a spectrum $\MTU(d)_{W_r}$ in an analogous way. In \cite{nguyen1}, we prove the following:

\begin{proposition}\label{heq-1}
	There is a commutative diagram where all vertical maps are homotopy equivalences.
	
	\begin{center}
		\begin{tikzcd}[column sep=large, row sep=large]
			\MTU(d-r)
			\arrow[r]  
			\arrow[d, leftarrow]& \MTU(d)	  \arrow[d, leftarrow] \\
			\MTU(d)_{V_r} \arrow[r] 	& \MTU(d)_{W_r} 
		\end{tikzcd}
	\end{center}
\end{proposition}

By using these homotopy equivalences in \cite{nguyen1}, we show that:
\begin{theorem}
	$\pi_{2d} (\MTU(d-r))$ is the even complex section cobordism group of $2d$ dimensional manifolds with $r$ complex sections. Elements of $\pi_{2d} (\MTU(d-r))$ are classes of almost complex manifolds with $r$ complex sections. Two classes are the same if there is a complex cobordism equipped with $r+1$ sections between them. The map $\pi_{2d} (\MTU(d-r))\ra \pi_{2d} (\MU)$ is the forgetful map. The image of the forgetful map is all cobordism classes which contain a manifold which can be equipped with $r$ complex sections.
\end{theorem}

The space $\MTU(d,1)$ can be given a simpler description which will be useful for computing the homotopy groups: (cf. \cite{bokstedt}[Proposition 2.6])  

\begin{theorem} \label{MTUd1}
	There is a homotopy equivalence:
	\[\MTU(d,1)\cong S^{\infty+2d} \vee \Sigma^{\infty+2d} BU(d)\]
	In particular,
	\[\pi_{q+2d}(\MTU(d,1))\cong \pi_q^s(S^{0}) \oplus \pi_{q}^s (BU(d))\]
\end{theorem}
\begin{proof}
	We observe that $\MTU(d,1)_{2n}\cong Th(p^*_{W_{\C,1}} U^\perp_{\C, d,n})/Th(p^*_{V_{\C,1}} U^\perp_{\C, d,n})$ where $p^*_{V_{\C,1}}$ is the projection $p^*_{V_{\C,1}}: V_1(U_{\C, d,n})\cong S(U_{\C, d,n})\ra G_\C(d,n)$ and $p^*_{W_{\C,1}}$ is the projection $p^*_{W_{\C,1}}: W_1(U_{\C, d,n})\cong D(U_{\C, d,n})\ra G_\C(d,n)$.
	For two vector bundles $V_1$ and $V_2$ over the same base,
	\[Th(V_1\oplus V_2)\cong Th(V_2|_{D(V_1)})/Th(V_2|_{S(V_1)}) \]
	Thus, \[\MTU(d,1)_{2n}\cong Th(U_{\C, d,n}\oplus U^\perp_{\C, d,n})\]
	Since, $U_{\C, d,n}\oplus U^\perp_{\C, d,n}$ is a trivial $d+n$ dimensional complex bundle,
	 \[\MTU(d,1)_{2n}\cong S^{2d+2n} \vee \Sigma^{2d+2n} G_\C(d,n)\] Letting $n\ra \infty$ completes the proof.
\end{proof}

We briefly describe the homology of the spectra $\MTU(d)$.
\begin{proposition}\label{homMTU}
	There is a Thom isomorphism:
	\[H_*(BU(d);\Z)\cong H_*(\MTU(d);\Z)\]
	and the map $H_*(\MTU(d);\Z)\ra H_*(\MU;\Z)$ is an injection. Moreover, the image of $H_*(\MTU(d))$ in $H_*(\MU;\Z)\cong \Z[B_1,B_2,...]$ is generated as a group by monomials of degree less than $d$.

\end{proposition}
\begin{proof}
	This result follows by dualizing the corresponding result, Proposition 2.4, for the cohomology of $\MTU(d)$ from \cite{nguyen1}.
\end{proof}

The lower homotopy groups of $\MTU(d)$ are trivial to compute \cite{bokstedt2}, \cite{galatius}.
\begin{lemma} \label{iso}
	The group $\pi_{q}(\MTU(d))\cong \pi_{q}(\MU)$ for $q\leq 2d$. 
\end{lemma}

There is an obstruction to lifting an element of $\pi_{2d}(\MU)$ to $\pi_{2d}(\MTU(d-r))$ which lands in the homotopy group $\pi_{2d}(\MTU(d,r))$. In \cite{nguyen1} we described the image of $\pi_*(\MTU(d))\otimes \Q\ra \pi_*(\MU)\otimes \Q$ as follows:

\begin{theorem}\label{ratobs1}
	Let $M^{2d}$ be a $d$-dimensional almost-complex manifold. The Chern characteristic classes $s_\omega(M^{2d})=0$ for all $\omega$ of length greater than $d-r$ if and only if the cobordism class $[M^{2d}]$ lifts to $\pi_{2d}(\MTU(d-r))\otimes \Q$. Geometrically, this implies there exists integer $c$ such that the cobordism class $c[M^{2d}]$ contains a manifold $N^{2d}$ with $r$ complex sections on $TN$.
\end{theorem}

 We will now describe the image of the map $\pi_*(\MTU(d))\ra \pi_*(\MU)$. As stated in the introduction, we will call the torsion in the cokernel, the torsion obstructions. Thus we determine which cobordism classes of manifolds can be equipped with complex sections. To do so, we will study the $p$-local Adams-Novikov spectral sequence for $\MTU(d)$.

\section{Computing the torsion obstruction}
This section will be devoted to identifying which elements of $\pi_*(\MU)$ can be lifted to $\pi_*(\MTU(d))$.  We will use the following notational conventions. Then let $B_i\in MU_*(\MU)\cong\pi_*(\MU)\otimes H_*(\MU)$, be the generators (as a $\pi_*(\MU)$ algebra). These $B_i$ correspond to the elements $b_i\in H_*(\MU)\cong \Z[b_1,b_2,...]$. There is an injective Hurewicz homomorphism $\pi_*(\MU)\ra H_*(\MU)$. By slight abuse of notation, we call the elements of $\pi_*(\MU)$ by their image in $H_*(\MU)$. We begin by stating the following results about $MU$ theory originally developed by \cite{novikov} and described in detail in \cite{ravenel} Appendix 2.1.

\begin{theorem}
	The pair $(\pi_*(\MU), MU_*(\MU))$ is a Hopf algebroid with the following structures:
	\begin{itemize}
		\item An augmentation map $\e: MU_*(\MU)\ra \pi_*(\MU)$, defined by $\e(B_i)=0$.
		\item A left unit $\eta_L: \pi_*(\MU)\ra MU_*(\MU)$, defined as the standard inclusion.
		\item A right unit, $\eta_R: \pi_*(\MU)\ra MU_*(\MU)$, corresponding to the Hurewicz homomorphism, defined by the formula:
		 \[\sum_i\eta_R(b_i) x^{i+1}= \sum_j B_j \left(\sum_k b_k^{i+1} \right)^{j+1}  \]
		\item A product inherited from $H_*(\MU)$.
		\item A coproduct defined by:
		\[\sum_i\Delta(B_i) = \sum_j \left(\sum_i B_i \right)^{j+1}\otimes B_j   \]
	\end{itemize}
 
\end{theorem}
It is typically convenient to localize at a prime $p$. The structure of $\MU$ localized at $p$ is described as follows. See \cite{ravenel} Appendix 2.1 for more details.

\begin{theorem}
	For each prime $p$, there is an associative ring spectrum $\mathbf{BP}$ which is a retract of $\MU_{(p)}$, moreover, $\MU_{(p)}$ splits as copies of $\mathbf{BP}$. The homotopy group $\pi_*(\mathbf{BP})$ is $\Z_{(p)}[v_1,v_2,...]$ where $v_i\in \pi_{2(p^i-1)}(\mathbf{BP})$. The pair $(\pi_*(\mathbf{BP}),BP_*(\mathbf{BP}))$ is a Hopf algebroid. We will write $BP_*(\mathbf{BP})=\pi_*(\mathbf{BP})[t_1,t_2,...]$ with element $t_i$ in degree $2(p^i-1)$. The algebroid has the structure:
	\begin{itemize}
		\item An augmentation map $\e: BP_*(\mathbf{BP})\ra \pi_*(\mathbf{BP})$, defined by $\e(t_i)=0$ and $\e(v_i)=v_i$.
		\item A left unit $\eta_L: \pi_*(\mathbf{BP})\ra BP_*(\mathbf{BP})$, defined as the standard inclusion.
		\item A right unit, $\eta_R: \pi_*(\mathbf{BP})\ra BP_*(\mathbf{BP})$, corresponding to the Hurewicz homomorphism, defined by the formula:
		\[\eta_R(m_i)= \sum_{j=1}^n m_j t_{i-j}^{p^i} \]
		\item A product inherited from $H_*(\mathbf{BP})$.
		\item A coproduct determined by:
		\[\sum_{i,j}m_i\Delta(t_j)^{p^i} = \sum_{i,j,k} \lambda_i t_j^{p^i}\otimes t_k^{p^{i+j}}\]
	\end{itemize}
	
	The basis $m_i$ comes from $[\C P^{p^i-1}]/p$. The generators $v_i$ are defined by 
	\[pm_n=\sum_{i=0}^{n-1} m_i v_{n-i}^{p^i}\]
	
\end{theorem}

We specifically note that $\eta_R(v_1)=v_1+2t_1$, and $\Delta(t_1)=t_1 \otimes 1+ 1\otimes t_1$.

	\begin{proposition}
		Let $X$ be a spectrum with torsion-free homology concentrated in even degrees. Let $E$ be either $\MU$ or $\mathbf{BP}$ for some prime $p$. Then the generalized theory $E_*(X)\cong \pi_*(E) \otimes_\Z H_*(X;\Z)$ and is torsion-free concentrated in even degrees.
	\end{proposition}
	\begin{proof}
		The Atiyah-Hirzebruch spectral sequence collapses, giving the stated result. \cite{atiyah1961vector}
	\end{proof}

		Using Proposition \ref{homMTU}, we immediately conclude:
	
	\begin{corollary}
		Let $D$ be either $\MU$ or $\mathbf{BP}$ for some prime $p$.
		The map
		\[D_*(\MTU(d))\ra D_*(\MU)=\pi_*(E)[B_1,B_2,...]\]
		is an injection, which is an isomorphism in degrees less than $2d+2$. Moreover, both modules are concentrated in even degrees. The image of $D_*(\MTU(d))$ maps onto polynomials of degree $d$. 
		\[D_*(\MU)\ra D_*(\overline{\MTU}(d))\] is the quotient map killing all polynomials of degree $d$ or less. These maps induce natural structures as $D_*(D)$ comodules.
	\end{corollary}

We will use the Adams-Novikov spectral sequence as our main tool to compute the cokernel of the map $\pi_*(\MTU(d))\ra \pi_*(\MU)$.
The $p$-local Adams-Novikov spectral sequence computes the homotopy groups of a finite type spectrum $X$ as follows, see \cite{novikov} or \cite{ravenel}.

\begin{theorem} 
	For a spectra of finite type $X$, there is a spectral sequence $\{E^{s,t}_k(X),d_k\}$ with differentials $d_k:E^{s,t}_k(X) \ra E^{s+r,t+r-1}_k(X)$: such that:
	\begin{enumerate}
		\item $E^{s,t}_2(X)\cong \Ext^{s,t}_{BP_*(\mathbf{BP})}(\pi_*(\mathbf{BP}), BP_*(X))$
		\item If $X$ is connective and $p$ local, there is a filtration
		\[ \pi_{t-s}(X)= F^{0,t-s} \sp ... \sp F^{s,t}\sp F^{s+1,t+1}\sp ...\] such that $E^{s,t}_\infty(X)=F^{s,t}/ F^{s+1,t+1}$.
	\end{enumerate}
\end{theorem}

Moreover, this construction is natural with respect to maps of spectra.
For an element in $E_q^{s,t}$, we will refer to $s$ as the resolution of the element and $t$ as the degree. The element, if it survives, corresponds to an element of $\pi_{t-s}$.
Consider the Adams-Novikov spectral sequence for $\MU_{(p)}$. Since $\MU_{(p)}$ splits as a copy of $\mathbf{BP}$'s, $E_2^{s,*}(\MU_{(p)})=0$ for $s>0$. So the spectral sequence collapses, and $E_2^{0,*}(\MU_{(p)})\cong\pi_*(\MU_{(p)})$. In particular, it is a free $\Z_{(p)}$ module. Moreover, by basic properties of $\Ext$ groups, \[E_2^{0,*}(\MTU(d)_{(p)})\cong \Hom(BP_*(\MTU(d)_{(p)}), \pi_*(\mathbf{BP})) \] The image of $\pi_*(\MTU(d)_{(p)})\ra \pi_*(\MU_{(p)})$ is concentrated in the 0 resolution. Thus, we first need to know the image of the map $E_2^{0,*}(\MTU(d))\ra E_2^{0,*}(\MU)$. This is the map \[ \Hom_{BP_*(\mathbf{BP})}(\pi_*(\mathbf{BP}),BP_*(\MTU(d)_{(p)})) \ra \Hom_{BP_*(\mathbf{BP})}(\pi_*(\mathbf{BP}),BP_*(\MU_{(p)}))\]

This map has cokernel which is a free $\Z_{(p)}$ module. Since all differentials are zero in the Adams-Novikov spectral sequence for $\MU_{(p)}$, the following proposition is concluded:

\begin{proposition}\label{diff}
	Fix prime $p$. Suppose $p[M]\in \pi_{*}(\MU)$ lifts to $\pi_*(\MTU(d))$. Then $[M]$ lifts to an element $x\in E_2^{0,*}(\MTU(d))$ via the map: \[\Hom_{BP_*(\mathbf{BP})}(\pi_*(\mathbf{BP}),BP_*(\MTU(d)_{(p)})) \ra \pi_*(\MU_{(p)})\] Moreover, $[M]$ lifts to $\pi_*(\MTU(d))$ if and only if there is no Adams differential originating from $x$. 
\end{proposition}
\begin{proof}
	Let $[M]\in \pi_*(\MU)$, where $p[M]\in \pi_{*}(\MU)$ lifts to $\pi_*(\MTU(d))$. We may view $[M]$ as an element $z\in E_2^{0,*}(\MU)$. We know $p[M]$ can be viewed as an element in $\pi_*(\MTU(d))$, which lifts to $y\in E_2^{0,*}(\MTU(d))$. By naturality of the Adams Novikov spectral sequence $y\mapsto pz$. Since the cokernel of $ E_2^{0,*}(\MTU(d))\ra  E_2^{0,*}(\MU)$ is free and the map is injective, there must be a unique element $x\mapsto z$. The element $x$ survives to the $E_\infty$ page if and only if all Adams differentials vanish on $x$.
\end{proof}
Thus, the differentials in the Adams-Novikov spectral sequence for $\MTU(d)$ correspond with torsion obstructions. We will use the cobar complex in order to compute the $E_2$ page. See Appendix 1 in \cite{ravenel} for more details. If we wished to extend this computation further, it would be beneficial to use a more sophisticated method. Recall that $BP_*(\mathbf{BP})$ has an action of $\pi_*(\mathbf{BP})$ given by multiplication by the image of $\eta_R$. Let $\overline{\Gamma}$ be the quotient of $BP_*(\mathbf{BP})$ the image of $\eta_L$.

\begin{definition}
	Let $(M,\psi)$ be a left $BP_*(\mathbf{BP})$ comodule. Define the cobar complex as the tensor product $C_{BP}^s(M)=\overline{\Gamma}^{\otimes s}\otimes_{\pi_*(\mathbf{BP})}M $.
	The differential in the complex is $d_1:C_{BP}^s(M)\ra C_{BP}^{s+1}(M)$, given by:
	\[d_1(\gamma_1 \otimes ... \otimes \gamma_s \otimes m)=\sum_{i=1}^s (-1)^i \gamma_1\otimes...\otimes \Delta(\gamma_i)\otimes...\otimes m + (-1)^{s+1}\gamma_1\otimes...\otimes \gamma_s \otimes \psi(m) \]
\end{definition}
\begin{proposition}
	The homology of the cobar complex is the Ext groups of $M$, specifically: \[H(C_{BP}^*(M))=\Ext_{BP_*(\mathbf{BP})}(\pi_*(\mathbf{BP}),M)\]
\end{proposition}

It will typically be convenient to view the cobar complex as the $E_1$ page of the Adams-Novikov spectral sequence where the differential corresponds to $d_1$. We observe that the $0$ line of the $E_2$ page of the Adams-Novikov spectral sequence is generated by primitives, i.e. elements of the form $\psi(x)=1\otimes x$ where $\psi$ is the coaction. The following proposition will be used to compute $E_2(\MTU(d))$

\begin{proposition}\label{shift}
	There is an isomorphism for $s\geq 1$:
	\[E_*^{s,t}(\overline{\MTU}(d))\cong E_*^{s+1,t}(\MTU(d))\] 
\end{proposition}
\begin{proof}
	By Theorem 2.3.4 \cite{ravenel}, there is an exact sequence of spectral sequences \[...\ra E_*^{s,t}(\MTU(d))\ra  E_*^{s,t}(\MU)\ra E_*^{s,t}(\overline{\MTU}(d))\ra E_*^{s+1,t}(\MTU(d))\ra ...\]
	Since $E_*^{s,t}(\MU)$ for $s>0$, the result follows.\end{proof}
	We can now prove Theorem \ref{torsobs}.
	\begin{theorem}\label{torsobsbody}
		The $p$-torsion obstruction to the existence of $r$ complex sections vanishes when $r<p^2-p$.
	\end{theorem}
	\begin{proof}
		Fix prime $p>2$. Consider the cobar complex for $\overline{\MTU}(d)$. Suppose we have some element $\beta\in E_2^{2,t}(\overline{\MTU}(d))$ for $t<2(p^2-p+d+1)$ This element is represented by the tensor sum $\sum_{k} \mu_k \otimes x_k \in \Gamma^{\otimes s} \otimes BP_*(\MTU(d))$ such that $\sum_{k} \mu_k \otimes x_k\in \Ker(d_1)$. The element $\sum_{k} \mu_k \otimes x_k$ has a term with maximal degree $x_k$. We will call this term $\mu_i \otimes x_i$. Since the sum $\sum_{k} \mu_k \otimes x_k\in \Ker(d_1)$, we see that $d_1(\mu_i)$ must be $0$. We know $\mu_i$ has degree less than $2(p^2-p)$ because $x_i$ has degree greater than or equal to $(2d+2)$. The element $\mu_i \otimes 1$ can be considered as an element in the cobar complex for the sphere spectrum, and since it is in $\Ker(d_1)$, it represents an element in $E_2(S^0)$. By \cite{zahler}[7.1], this element must be zero. Thus $\mu_i \otimes 1$ is in the image of $\sum_j \tilde{\mu}_j \otimes \nu_j$ where $\tilde{\mu}_j\in \overline{\Gamma}^{\otimes s-1}$ and $\nu_j\in \pi_*(BP)$. Then, we observe that $d_1(\sum_j \tilde{\mu}_j \otimes \nu_jx_i)= \mu_i\otimes x_i $. Thus, $\sum_{k\neq i} \mu_k \otimes x_k$ also represents $\beta$. Repeating this process, $\beta$ can be represented by $0$, and thus $\beta=0$. (We choose the letter $\beta$ suggestively, as the first potentially non zero $\beta$ corresponds to $\beta_1$ constructed in Definition 1.3.14 \cite{ravenel}). By Proposition \ref{shift}, $E_2^{3,t}(\MTU(d))=0 $ for $t<2(p^2-p+d+1)$. Thus, $d_3: E_3^{0,t}\ra E_3^{3,t+2}$ is trivial for $t<2(p^2-p+d)$. The same argument shows that all higher differentials are zero. (In fact, the bound is even stronger.) By Proposition \ref{diff}, the $p$-torsion obstruction to $r<p^2-p$ complex sections vanishes.
	\end{proof}	
	Suppose $t$ is such that $E^{s,q}_2(\Sigma^{\infty}S^{0})=0$ for all $q\leq t$. In particular for $s=2$, this is equivalent to supposing $t<2(p^2-p)$.
	\begin{remark}
		This result is not sharp, and we can improve it significantly by closer analysis. However, this bound is sufficient for this paper, since it will show the vanishing of all $p>2$ torsion obstructions for $r<6$.
	\end{remark}
	
	We now focus on the case $p=2$. For the rest of the paper all homotopy groups are assumed to be 2-primary. We will need to compute the homotopy groups of $\overline{\MTU}(d)$. To do so we remind ourselves of the Adams-Novikov spectral sequence for the sphere spectrum: \cite{zahler}
	
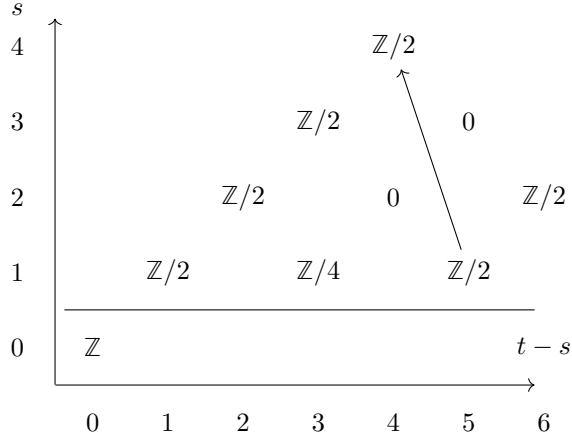
	\begin{center}
		\begin{tikzpicture}
			\node (v1) at (-0.12,0) {};
			\node (v0) at (0,-0.12) {};
			\node (v2) at (6.5,0) {};
			\draw[->]  (v1) edge (v2);
			\node (v3) at (0,5) {};
			\draw[->] (v0) edge (v3);
			\node at (-0.5,0.5) {0};
			\node at (-0.5,1.5) {1};
			\node at (-0.5,2.5) {2};
			\node (v4) at (0,1) {};
			\node (v5) at (6.5,1) {};
			\draw  (v4) edge (v5);
			\node at (0.5,-0.5) {$0$};
			\node at (2.5,-0.5) {$2$};
			\node at (4.5,-0.5) {$4$};
			\node at (1.5,1.5) {$\Z/2$};
			\node at (4.5,2.5) {$0$};
			\node at (2.5,2.5) {$\Z/2$};
			\node at (-0.5,3.5) {3};
			
			\node at (3.5,3.5) {$\Z/2$};
			\node at (3.5,1.5) {$\Z/4$};
			\node at (0.5,0.5) {$\Z$};

			\node at (-0.5,4.5) {$4$};
			\node (v7) at (4.5,4.5) {$\Z/2$};
			\node (v6) at (5.5,1.5) {$\Z/2$};
			\draw[->]  (v6) edge (v7);
			\node at (-0.5,5) {$s$};
			\node at (6.5,0.5) {$t-s$};
			\node at (6.5,2.5) {$\Z/2$};
			\node at (5.5,3.5) {0};
			\node at (5.5,-0.5) {5};
			\node at (6.5,-0.5) {6};
			\node at (1.5,-0.5) {1};
			\node at (3.5,-0.5) {3};
		\end{tikzpicture}
	\end{center}
	In particular, we will call the generators of $\pi_1^s(S^0)$ and $\pi_3^s(S^0)$, $h_1$ and $h_2$ respectively. These can be represented by the Hopf maps.
	\begin{theorem}\label{homgroups}
		The homotopy groups of $\overline{\MTU}(d)$ for $d\geq3$ are given by: 
		\[\begin{array}{cccccc}
			q & 2d & 2(d+1) & 2(d+2) & 2(d+3) & 2(d+4) \\
			d\equiv 0\ (\tn{mod}\ 2)& \Z & \Z/2 &  \Z/2 \oplus \Z^{2}& ?  & \Z^{4}\\
			d\equiv 1\ (\tn{mod}\ 2)& \Z & 0 & \Z^{2} & 0 & \Z^{4}\\
		\end{array} \]
	\end{theorem}
	\begin{proof}
		 The coaction $\psi:BP_*(\MTU(d)) \ra BP_*(\MU)\otimes_{\pi_*(\mathbf{BP})} BP_*(\MTU(d))$ is the natural quotient of the coaction in $BP_*(\MU)$ induced by the map $\MU\ra \MTU(d)$. Note that the right action of $\pi_*(\mathbf{BP})$ on $BP_*(\MU)$ is given by multiplication by the image of the Hurewicz homomorphism $\eta_R$. We compute that $\eta_R(v_1)=v_1+2t_1$. (See \cite{ravenel}, appendix 2.2.)
		
		We start by computing the following coactions in $BP_*(\MU)$. 
		\begin{align*}
			\psi(B_1)&=1\otimes B_1+t_1\otimes 1\\
			\psi(B_2)&=1\otimes B_2 +(2t_1)\otimes B_1+t_1^2\otimes 1
		\end{align*} 
		We can then compute:.		
		\begin{align*}
			\psi(B_1^{d+1})&=1\otimes B_1^{d+1}\\
			\psi(B_2B_1^d)&=2t_1 \otimes B_1^{d+1}+1 \otimes B_2B_1^{d}\\
			\psi(B_1^{d+2})&=1\otimes B_1^{d+2}+(d+2)t_1\otimes B_1^{d+1}\\
			\psi(v_1B_1^{d+1})&=v_1\otimes B_1^{d+1}=1\otimes v_1B_1^{d+1}-2t_1\otimes B_1^{d+1}
		\end{align*} 
		
		We have two cases, if $d$ is even, then the cokernel of $d_1:E_1^{0,2d+4}\ra E_1^{1,2d+4}$ is $\Z/2$ generated by $t_1\otimes B_1^{d+1}$ and if $d$ is odd, the cokernel is zero.	Next we compute the cokernel of $d_1:E_1^{1,2d+6}\ra E_1^{2,2d+6}$.		
		\begin{align*}
			&d_1(t_1\otimes B_2B_1^d)=2 t_1\otimes t_1 \otimes \beta_1^{d+1} & d_1&(t_1 \otimes B_1^{d+2})= (d+2)t_1\otimes t_1 \otimes \beta_1^{d+1}\\ &d_1(t_1 \otimes v_1B_1^{d+1})= -2 t_1\otimes t_1 \otimes \beta_1^{d+1}& d_1&(t_1^2 \otimes B_1^{d+1})=-2 t_1\otimes t_1 \otimes \beta_1^{d+1}
		\end{align*}
		Once again, if $d$ is even, then the cokernel of $d_1:E_1^{1,2d+6}\ra E_1^{2,2d+6}$ is $\Z/2$, and if $d$ is odd, the cokernel is zero. For similar reasons, $d_1:E_1^{2,2d+8}\ra E_1^{3,2d+8}$ has cokernels $\Z/2$ or $0$ when $d$ is even or odd respectively.	We now compute the image of the map $d_1:E_1^{0,2d+6}\ra E_1^{1,2d+6}$. We will write this as a table giving of coefficients in the basis $ t_1\otimes B_2B_1^d, t_1 \otimes B_1^{d+2}, t_1 \otimes v_1B_1^{d+1}, t_1^2 \otimes B_1^{d+1} $		
		\[\begin{array}{c|c|c|c|c}
			B_3B_1^d & 3 & 0 & -2 & 5
			 \\
			B_2^2 B_1^{d-1}& 4 & 0 & 0 & 4 \\
			B_2 B_1^{d+1}& (d+1) & 2 & 0 & 2d+3 \\
			B_1^{d+3}& 0 & (d+3) & 0 & \frac{1}{2}(d+3)(d+2) \\
			v_1B_2B_1^d& -2 & 0 & 2 & -4 \\
			v_1B_1^{d+2}& 0 & -2 & (d+2) & 2(d+2) \\
			v_1^2B_1^{d+1}& 0 & 0 & -4 & 4 
		\end{array}\]
		The kernel of $d_1:E_1^{0,2d+6}\ra E_1^{1,2d+6}$ is generated by elements \[a_1t_1\otimes B_2B_1^d+a_2 t_1 \otimes B_1^{d+2}+a_3 t_1 \otimes v_1B_1^{d+1}+a_4 t_1^2 \otimes B_1^{d+1}\]
		where $2a_1+(d+2)a_2-2(a_3+a_4)=0$. Thus, we have two cases for $E_2^{1,2d+4}$: if $d$ is even, $E_2^{1,2d+4}=\Z/2$ and if $d$ is odd, $E_2^{1,2d+4}=0$. Similarly we can compute  $d_1:E_1^{1,2d+8}\ra E_1^{2,2d+8}$ in the basis: $ t_1 \otimes t_1\otimes B_2B_1^d,  t_1 \otimes t_1 \otimes B_1^{d+2},  t_1 \otimes t_1 \otimes v_1B_1^{d+1},\\ t_1 \otimes t_1^2 \otimes B_1^{d+1}, t_1^2 \otimes t_1 \otimes B_1^{d+1}$. 
		
		\[\begin{array}{c|c|c|c|c|c}
			t_1 \otimes B_3B_1^d & 3 & 0 & -2 & 5& 0\\
			t_1 \otimes B_2^2 B_1^{d-1}& 4 & 0 & 0 & 4 & 0\\
			t_1 \otimes B_2 B_1^{d+1}& (d+1) & 2 & 0 & 2d+3 & 0\\
			t_1 \otimes B_1^{d+3}& 0 & (d+3) & 0 & \frac{1}{2}(d+3)(d+2) &0 \\
			t_1 \otimes v_1B_2B_1^d& -2 & 0 & 2 & -4 &0 \\
			t_1 \otimes v_1B_1^{d+2}& 0 & -2 & (d+2) & 2(d+2) & 0\\
			t_1 \otimes v_1^2B_1^{d+1}& 0 & 0 & -4 & 4 &0 \\
			t_1^2 \otimes B_2B_1^d & -2& 0&0 &0 & 2 \\
			t_1^2 \otimes B_1^{d+2} &0 &-2 & 0& 0& (d+2)\\
			t_1^2 \otimes v_1B_1^{d+1} & 0& 0& -2&0 & -2 \\
			t_2 \otimes B_1^{d+1} & 0& 0& -1 & 3 & 2 \\
			t_1^3 \otimes B_1^{d+1} &0 &0 & 0& 3 & 3
		\end{array}\]
We find that $E_2^{2,2d+6}$ is always zero, and the Adams chart for $d$ even is:
		
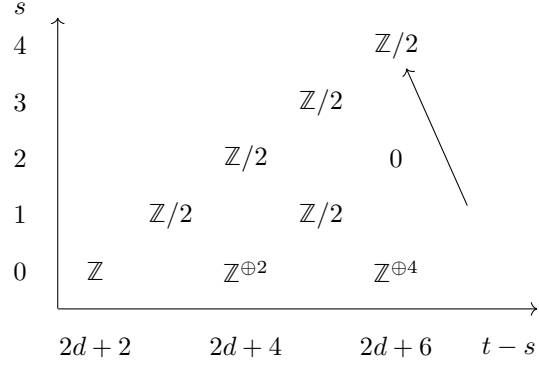
\begin{center}
			\begin{tikzpicture}
			\node (v1) at (-0.12,0) {};
			\node (v0) at (0,-0.12) {};
			\node (v2) at (6.5,0) {};
			\draw[->]  (v1) edge (v2);
			\node (v3) at (0,4) {};
			\draw[->] (v0) edge (v3);
			\node at (-0.5,0.5) {0};
			\node at (-0.5,1.25) {1};
			\node at (-0.5,2) {2};
			
			\node at (0.5,-0.5) {$2d+2$};
			\node at (2.5,-0.5) {$2d+4$};
			\node at (4.5,-0.5) {$2d+6$};
			\node at (1.5,1.25) {$\Z/2$};
			\node at (4.5,2) {$0$};
			\node at (2.5,2) {$\Z/2$};
			\node at (-0.5,2.75) {3};
			\node at (3.5,2.75) {$\Z/2$};
			\node at (3.5,1.25) {$\Z/2$};
			\node at (0.5,0.5) {$\Z$};
			\node at (2.5,0.5) {$\Z^{\oplus 2}$};
			\node at (4.5,0.5) {$\Z^{\oplus 4}$};
			\node at (-0.5,3.5) {$4$};
			\node (v7) at (4.5,3.5) {$\Z/2$};
			\node (v6) at (5.5,1.25) {};
			\draw[->]  (v6) edge (v7);
			\node at (-0.5,4) {$s$};
			\node at (6,-0.5) {$t-s$};
		\end{tikzpicture}
		
		\end{center}
		The non-trivial $d_3$ differential can be observed by looking at the map of spectra $S^{\infty+2d}\ra \overline{\MTU}(d)$ given by the generator of $\pi_{2d}(\overline{\MTU}(d))$. There may be an additional differential killing $E_2^{2,2d+8}$ but since this element is in odd degree, we will not need to know it. The Adams-Novikov spectral sequence for $d$ odd is trivial in this range.
	\end{proof}
		\begin{remark}
			The vanishing of the torsion obstruction to $r$ complex sections does not imply the vanishing of the torsion obstruction to $r-1$ complex sections, thus each result must be proved separately. 
		\end{remark}
		We use these Adams charts and homotopy groups to prove:
\begin{theorem}\label{3vf2}
	There is no torsion obstruction to the existence of 2 or 3 linearly independent complex sections on $M^{2d}$ for $d\geq 6$.
\end{theorem}
\begin{proof}
	Since $\pi_{2d+6}(\overline{\MTU}(d))$ is free abelian, there is no possible target for a torsion obstruction to $3$ complex sections. Similarly, there is no possible target for a torsion obstruction to $2$ complex sections when $d$ is odd. In the case $d$ is even, there could be an obstruction to $2$ complex sections. However, we recall from \cite{KashiwabaraZare} that $S^{\infty+2d}$ splits off $\MTU(d)$ for $d$ even. The element in $E_2^{3,2d+3}$ comes from this split $S^{\infty+2d}$ (specifically, it can be identified with $h_1^3$), and so cannot be the target of a differential.
\end{proof}

\begin{theorem}\label{4vf2}
	There is no torsion obstruction to $4$ complex sections when $d>6$ is odd.
\end{theorem}
\begin{proof}
	Let $M\in \pi_{2d}(\MU)$ with vanishing rational obstruction. Recall that the image of $M$ in the group $\pi_{2d}(\overline{\MTU}(d-4))$ is the obstruction to $4$ complex sections. Thus the image is torsion. By Theorem \ref{3vf2}, there is no torsion obstruction to $3$ complex sections so the image of $M$ in $\pi_{2d}(\overline{\MTU}(d-3))$ is zero. Thus the obstruction to $4$ complex sections lifts to $\pi_{2d}(\MTU(d-3,1))\cong \pi_{2d}^s(S^{2d-6}) \oplus \pi_6^s(BU(d-3))$. 
	\begin{center}
		\begin{tikzpicture}
		\node (v1) at (-1.5,0) {$\pi_{2d}(\MTU(d-4))$};
		\node (v2) at (2,0) {$\pi_{2d}(\MU)$};
		\node (v3) at (5.5,0) {$\pi_{2d}(\overline{\MTU}(d-4))$};
		\node (v6) at (-1.5,-1) {$\pi_{2d}(\MTU(d-3))$};
		\node (v7) at (2,-1) {$\pi_{2d}(\MU)$};
		\node (v4) at (5.5,-1) {$\pi_{2d}(\overline{\MTU}(d-3))$};
		\node (v5) at (5.5,1) {$\pi_{2d}(\MTU(d-3,1))$};
		\draw [->] (v1) edge (v2);
		\draw [->] (v2) edge (v3);
		\draw [->] (v3) edge (v4);
		\draw [->] (v5) edge (v3);
		\draw [->] (v1) edge (v6);
		\draw [->] (v6) edge (v7);
		\draw [->] (v7) edge (v4);
		\draw [double] (v7) edge (v2);
	\end{tikzpicture}
	\end{center}
	
	The group $\pi_6^s(BU(d-3))$ is torsion free. The torsion element in $\pi_6^s(S^0)$ is $h_2\cdot h_2$. The map $\pi_{2d}^s(S^{2d-6})\ra \pi_{2d}(\MTU(d-3))$ is a $\pi_*^s(S^0)$ module map. In particular, it must take $h_2$ to $0$ if $d$ is odd and we conclude that $h_2\cdot h_2\mapsto h_2\cdot 0=0$. Thus the torsion obstruction to $M$ having $4$ complex sections vanishes.
\end{proof}

\begin{remark}
	For $d$ even, there is a potential obstruction of order $2$. We expect that this obstruction is non vanishing and, based on similar results in \cite{atiyahdupont} and \cite{bokstedt}, we expect that the obstruction is a divisibility condition on the signature of the manifold.
\end{remark}

We can give a bound on the possible torsion obstruction to the existence of $r$ complex sections. Using the following theorem of Segal \cite{Segal}.

\begin{theorem}
	There are no non-zero differentials originating in the bottom row of the Adams-Novikov spectral sequence for $\MTU(1)\cong \Sigma^{\infty-2} \C P^\infty$.
\end{theorem}

\begin{corollary}
	There is no torsion obstruction to the existence of $d-1$ sections on $[M^{2d}]$.
\end{corollary}

In particular, we can find an element of $N^{2d}\in\pi_{2d}(\MTU(1))$ with all characteristic classes zero except $s_{d}(N^{2d})=(d+1)!/B_{2d}$ where $B_{2d}$ is the Bernoulli number \cite{stongcong}. From Milnor's characterization of the generators of the complex cobordism ring, the following proposition immediately follows. These numbers are known to play a role in the theory of characteristic classes cf. \cite{Adamschern}. \begin{proposition}\label{goodgen}
	For every $d$, define \[a_{d}=\left\lbrace\begin{array}{cc}
		(d+1)!/B_{2d} & d\neq p^i-1 \\
		(d+1)!/(pB_{2d})& d=p^i-1 \\
	\end{array} \right. \]
	Then for any set of multiplicative generators $M^{2d}\in \Omega_{2d}^U$, we can write the cobordism class as $N^{2d}= \tilde{a}_{d} M^{2d} + \tn{decomposables} $
\end{proposition}

	For any partition $I=\{i_1,...,i_n\}$, define $a_I=\tilde{a}_{i_1}...\tilde{a}_{i_k}$. Let $M^{2i}\in \Omega_{2i}^U$ be any collection of multiplicative generators and write $M^I= M^{2i_1}\times ... \times M^{2i_k}$. Define a partial order on partitions of $d$ by \[\{i_1,...,i_k\}> \{j_{1,1},...,j_{1,l_1},j_{2,1},...,j_{2,l_1},.... ,j_{k,1},...,j_{k,l_k}\}\]
	if $i_n=\sum_{m=1}^{l_m} j_{n,m}$. 
\begin{theorem}
	Let $K^{2d}=M^I + \sum_{J<I} b_JM^J $ be a manifold with vanishing rational obstruction to $r$ complex sections. Then $\tn{lcm}(a_J\mid J<I)[K^{2d}]$ contains a manifold which can be equipped with $r$ complex sections.
\end{theorem}
\begin{proof}
	Call $C=\tn{lcm}(a_J\mid J<I)$.	By Proposition \ref{goodgen}, we can write \[CK^{2d}=\frac{C}{a_J}N^I+C\sum_{J<I} b_J'M^J\]
	for new constants $b_J'$. Choose maximal $J$ with $b_J'\neq 0$. Then we can rewrite $CM^J$ as $\frac{C}{a_J}\sum_{L<J} N^J$. By inducting down on the partial order, we can write 
	\[CK^{2d}=c_IN^I+\sum_{J<I} c_JN^J\]
 	for some integers $c_J$. Note that $N^J$ has $s_J(N^J)\neq 0$ and all other characteristic classes equal to 0. So, in this sum, $c_J$ must be zero for all $J$ of length greater than or equal to $d-r+1$ because we are assuming the rational obstruction vanishes. Note that by construction $N^{2i}$ has $i-1$ complex sections, so $N^J$ has $d-l(J)\geq r$ complex sections. Thus $a_IK^{2d}$ also can be equipped with $r$ complex sections.
\end{proof}

\begin{remark}
	We expect that the Bernoulli numbers are measuring unexpected obstructions to complex sections, by comparing the formula in Proposition $\ref{goodgen}$ with the formula for the primitives in \cite{Segal}. Thus, we conjecture that they relate to the presence of non trivial differentials.
\end{remark}

This result is not sharp, but it allows us to construct cobordism classes containing a manifold with $r$ complex sections for any $r$.

	\newpage

\bibliography{ref}{}
\bibliographystyle{plain}
\end{document}